\documentclass[11pt]{amsart}
\usepackage{amsmath, amsthm, amssymb, latexsym,url}
\usepackage{hyperref}
\usepackage[mathscr]{eucal}

\author[D. Airey]{Dylan Airey}
\address[D. Airey]{
Department of Mathematics, University of Texas at Austin, 2515 Speedway, Austin, TX 78712-1202, USA}
\email{dylan.airey@utexas.edu}

\author[B. Mance]{Bill Mance}
\address[B. Mance]{Department of Mathematics, University of North Texas, General Academics Building 435, 1155 Union Circle,  \#311430, Denton, TX 76203-5017, USA}
\email{mance@unt.edu}

\author[J. Vandehey]{Joseph Vandehey}
\address[J. Vandehey]{Department of Mathematics, University of Georgia at Athens, Boyd graduate studies research center, Athens, GA 30606 USA}
\email{vandehey@uga.edu}

\thanks{Research of the first and second authors is partially supported by the U.S. NSF grant DMS-0943870.  }

\title{Normal number constructions for Cantor series with slowly growing bases}
\date{\today}

\newtheorem{thm}{Theorem}[section]

\newtheorem{lem}[thm]{Lemma}

\newtheorem{defn}[thm]{Definition}

\newcommand{\labeq}[1]{\label{eq:#1}}
\newcommand{\refeq}[1]{(\ref{eq:#1})}

\newcommand{\reft}[1]{Theorem~\ref{thm:#1}}

\newcommand{\N}{\mathbb{N}}
\newcommand{\br}[1]{\left \{ #1 \right \}}
\newcommand{\NQ}{\mathscr{N}(Q)}
\newcommand{\DNQ}{\mathscr{DN}(Q)}
\newcommand{\RNQ}{\mathscr{RN}(Q)}
\newcommand{\floor}[1]{\left\lfloor #1 \right\rfloor} 
\newcommand{\ceil}[1]{\left\lceil #1 \right\rceil}
\newcommand{\pr}[1]{\left ( #1 \right )}

\allowdisplaybreaks

\begin{document}

\maketitle

\begin{abstract}
Let $Q=(q_n)_{n=1}^\infty$ be a sequence of bases with $q_i\ge 2$. In the case when the $q_i$ are slowly growing and satisfy some additional weak conditions, we provide a construction of a number whose $Q$-Cantor series expansion is both $Q$-normal and $Q$-distribution normal. Moreover, this construction will result in a computable number provided we have some additional conditions on the computability of $Q$, and from this construction we can provide computable constructions of numbers with atypical normality properties.
\end{abstract}

\section{Introduction}\label{section:introduction}

A real number $x$ has a unique base $b$ expansion of the form
\begin{equation}\label{eq:bdef}
x=a_0+\sum_{n=1}^\infty \frac{a_n}{b^n}
\end{equation}
where $a_0=\floor{x}$ and the digits $a_n$ satisfy $a_n \in \{0,1,2,\dots,b-1\}$ and $a_n\neq b-1$ infinitely often. This number is said to be normal to base $b$ if for every finite sequence $(c_j)_{j=1}^k$ with $c_j \in \{0,1,2,\dots,b-1\}$, we have
\[
\lim_{n\to \infty} \frac{\#\{1\le i \le n \mid c_j =a_{i+j-1}, 1\le i \le k\}}{n} = \frac{1}{b^k}.
\]
This definition says that a number is normal when each string of digits appears with the frequency one would expect if the digits were chosen at random. Equivalently, one could say that the sequence $(b^k x)_{k=0}^\infty$ is uniformly distributed modulo $1$.

Although almost all real numbers are normal to base $b$, very few examples of such numbers are known, and those examples that are known are numbers that were explicitly constructed to be normal. One of the very first such constructions was due to Champernowne \cite{Champernowne}, who showed that the number
\[
0.12345678910111213141516\dots,
\]
formed by concatenating all the integers, was normal to base $10$. 

There are, of course, many different ways of representing a real number, such as continued fraction expansions and beta expansions, each with their own definitions of normality. Here, we are interested in the $Q$-Cantor series expansion. The study of normal numbers and other statistical properties of real numbers with respect to large classes of Cantor series expansions was  first done by P. Erd\H{o}s and A. R\'{e}nyi in \cite{ErdosRenyiConvergent} and \cite{ErdosRenyiFurther}, by A. R\'{e}nyi in \cite{RenyiProbability}, \cite{Renyi}, and \cite{RenyiSurvey}, and by P. Tur\'{a}n in \cite{Turan}.

The $Q$-Cantor series expansions, first studied by G. Cantor in \cite{Cantor},
are a natural generalization of the $b$-ary expansions.\footnote{G. Cantor's motivation to study the Cantor series expansions was to extend the well known proof of the irrationality of the number $e=\sum 1/n!$ to a larger class of numbers.  Results along these lines may be found in the monograph of J. Galambos \cite{Galambos}. } 
A basic sequence is a sequence of integers greater than or equal to $2$.
Given a basic sequence $Q=(q_n)_{n=1}^{\infty}$, the {\it $Q$-Cantor series expansion} of a real number $x$  is the (unique)\footnote{Uniqueness can be proven in the same way as for the $b$-ary expansions.} expansion of the form
\begin{equation} \labeq{cseries}
x=E_0+\sum_{n=1}^{\infty} \frac {E_n} {q_1 q_2 \cdots q_n}
\end{equation}
where $E_0=\floor{x}$ and $E_n$ is in $\{0,1,\cdots,q_n-1\}$ for $n\geq 1$ with $E_n \neq q_n-1$ infinitely often. We abbreviate \refeq{cseries} with the notation $x=E_0.E_1E_2E_3\cdots$ w.r.t. $Q$.

Definitions of normality for $Q$-Cantor series require a few more definitions. Given a block of digits $B=[B_1,B_2,B_3,\dots, B_k]$, define
\[
N_n^Q(B,x) = \#\{1 \le i\le n \mid E_{i+j-1} = B_j, 1\le j \le k\},
\]
so that $N_n^Q(B,x)$ counts the number of times a given block of digits appears in the $Q$-cantor expansion for $x$ up to the $n$th place. Moreover let
\[
I_i(B) = \begin{cases}
1, & B_j<q_{i+j-1}, \quad 1\le j \le k,\\
0, & \text{otherwise,}
\end{cases}
\]
so that $I_i(B)$ detects whether or not the digit block $B$ can even occur at the $i$th place in the $Q$-Cantor expansion for some point $x$. 

We let $|B|=k$ denote the length of the block $B$. For a block $B$ of length $k$ define
\[
Q_n(B) = \sum_{i=1}^n \frac{I_i(B)}{q_i q_{i+1} \dots q_{i+k-1}},
\]
which may be interpreted as the expected number of times to see the block $B$ in the first $n$ digits of a $Q$-Cantor series expansion if every digit $E_i$ is chosen at random from the set $\{0,1,\dots,q_i-1\}$. 
We also define
$$
T_{Q,n}(x) = q_n q_{n-1} \cdots q_1 x \pmod{1}.
$$

A real number $x$  is {\it $Q$-normal} if\footnote{We choose to take a slightly different definition for $Q$-normality than is used elsewhere in the literature. Our definition is more appropriate for bounded basic sequences.} for all blocks $B$ regardless of length such that $\lim_{n \to \infty} Q_n(B) = \infty$, we also have
$$
\lim_{n \rightarrow \infty} \frac {N_n^Q (B,x)} {Q_n(B)}=1.
$$
Let $\NQ$ be the set of $Q$-normal numbers.  As with the definition for base $b$ normality, this says, in essence, that the number of times a block of digit appears is the expected frequency if the digits were chosen at random. In fact, if we let $q_i = b$ for all $i$, this definition is precisely the definition for a base $b$ normal number.
The real number $x$ is {\it $Q$-ratio normal} (here we write $x \in \RNQ$) if for all blocks $B_1$ and $B_2$ of equal length where $\lim_{n \to \infty} Q_n(B_1) =\lim_{n \to \infty} Q_n(B_2) = \infty$ we have
$$
\lim_{n \to \infty} \frac {N_n^Q (B_1,x)} {N_n^Q (B_2,x)}=1.
$$
A real number~$x$ is {\it $Q$-distribution normal} if
the sequence $(T_{Q,n}(x))_{n=0}^\infty$ is uniformly distributed mod $1$. Again, if $q_i=b$ for all $i$, this definition is precisely the equivalent definition for a base $b$ normal number. Let $\DNQ$ be the set of $Q$-distribution normal numbers.  The relationship between $\NQ, \RNQ$, and $\DNQ$ is discussed in \cite{AireyManceHDDifference} and \cite{ppq1} but is not fully understood: for example, unlike for base $b$ expansions, there exist $Q$ such that $\NQ$ and $\DNQ$ are not the same set, although it is not known for which $Q$ this holds.

There are a number of classical results about base $b$ normality which do not yet have analogues for $Q$-Cantor expansion normality. Even the simplest question, asking for an example of a $Q$-normal number for any reasonable $Q$, is unanswered in many cases. Altomare and Mance \cite{AlMa} and Mance independently \cite{Mance} started with a set of data satisfying certain conditions and used this to generate both a sequence $Q$ and a number $x$ that was $Q$-normal. In their constructions, the sequence $Q$ was constant for very long stretches at a time. Most other constructions that have been found thus far, such as those in \cite{AireyManceVandehey}, also put very stringent restrictions on what $Q$ are allowed. Perhaps the most general $Q$-normal construction comes from \cite{AireyMancePeriodic}: there, the authors show that if $Q$ is eventually periodic, then there is some integer $b$ such that being $Q$-normal is equivalent to being base $b$ normal, and thus constructions of base $b$ normal numbers give $Q$-normal numbers in this case.

The first main result of this paper is the following, which provides a $Q$-normal number construction for a much broader set of basic sequences $Q$.

\begin{thm}\label{thm:main}
Let $Q$ be a basic sequence that satisfies the following two conditions:
\begin{itemize}
\item $Q$ is slowly growing; that is, if we let $q(n)=\max_{i\le n} q_i$, then $q(n) = n^{o(1)}$; and,
\item For any block of digits $B$ such that $\lim_{n\to \infty} Q_n(B) = \infty$ we have
\[
\lim_{n\to \infty} \frac{Q_n(B)}{n \log q(n)/\log n} = \infty.
\]
\end{itemize}
Then the number $x_Q$ constructed in Section \ref{sec:construction} is both $Q$-normal and $Q$-distribution normal.
\end{thm}

The number $x_Q$ that we construct is an explicit example. To define what we mean by an explicit example, we bring in some definitions from recursion theory.  A real number $x$ is \textit{computable} if there exists $b \in \mathbb{N}$ with $b\geq 2$ and a total recursive function $f: \mathbb{N} \to \mathbb{N}$ that calculates the digits of $x$ in base $b$. A sequence of real numbers $(x_n)$ is \textit{computable} if there exists a total recursive function $f: \mathbb{N}^2 \to \mathbb{Z}$ such that for all $m,n$ we have that $\frac{f(m,n)-1}{m} < x_n< \frac{f(m,n)-1}{m}$. 

M. W. Sierpi\'{n}ski gave an example of an absolutely normal number that is not computable in \cite{Sierpinski}.
The authors feel that examples such as M. W. Sierpi\'{n}ski's are not fully explicit since they are not computable real numbers, unlike Champernowne's number. 
A. M. Turing gave the first example of a computable absolutely normal number in an unpublished manuscript.  This paper may be found in his collected works \cite{Turing}. See \cite{BecherFigueiraPicchi} by V. Becher, S. Figueira, and R. Picchi for further discussion.

We will also show the following: 
\begin{thm} \label{thm:comp}
If $Q$ satisfies the conditions of Theorem \ref{thm:main} and $(q_n)_{n=1}^\infty$ and $(q(n))_{n=1}^\infty$ are computable sequences of integers, then the $x_Q$  is computable.
\end{thm}

In \cite{ppq1} the second author showed that for any basic sequence $Q$ that is infinite in limit such that $Q_n(B) \to \infty$ for each admissable block $B$ the set $\NQ \setminus \DNQ$ is non-empty. He also showed that $\RNQ \setminus \NQ$ is non-empty only assuming $Q$ is infinite in limit. In \cite{AireyManceHDDifference} the first and second authors improved this result and showed that if $Q$ is infinite in limit, the set $\RNQ \cap \DNQ \setminus \NQ$ has full Hausdorff dimension.

Along these lines will be able to provide constructions of computable real numbers that are in sets such as $\NQ \backslash \DNQ$ and $\RNQ \cap \DNQ \backslash \NQ$ under the same or slightly stronger assumptions than those of Theorem \ref{thm:comp}.

\subsection{Notations}

We will use asymptotic notations with their standard meaning. By $f(x)=O(g(x))$ or, equivalently, $f(x) \ll g(x)$, we mean that there is some constant $C$ such that $|f(x)|\le C g(x)$. By $f(x)\asymp g(x)$, we mean $f(x) = O(g(x))$ and $g(x)=O(f(x))$. By $f(x) = o(g(x))$, we mean that $\lim_{x\to \infty} f(x)/g(x) = 0$. By $f(x) \sim g(x)$, we mean that $f(x) = g(x)(1+o(1))$ or, equivalently, $\lim_{x\to \infty} f(x)/g(x) =1$.

\section{The construction}\label{sec:construction}

We need some additional definitions. Given two blocks of integers $A=[a_1,a_2,\dots,a_k]$ and $B=[b_1,b_2,\dots,b_k]$ (which could be blocks of digits, $E_i$ or blocks of bases, $q_i$), we say that $A<B$ if $a_i < b_i$ for $1\le i \le k$ (and we make an analogous definition for $A\le B$).

For a given integer $r$, let $n_r$ denote the smallest integer $n$ such that $(q(n)^2+1)^r\le n$. By the assumption that $q_n=n^{o(1)}$, the integer $n_r$ always exists. Let $(N_r)_{r=0}^\infty$ be an increasing sequence of non-negative integers defined so that $N_1=0$ and all the $N_{r+1}$'s are defined inductively as being the greatest integer less than  $n_{r+1}$ such that $N_{r+1}-N_{r}$ is divisible by $r$. In particular, we will have $N_{r} =n_r+O(r)$.

Divide the bases of $Q$ from the $N_{r}+1$st base to the $N_{r+1}$th base into $(N_{r+1}-N_{r})/r$ blocks of $r$ consecutive bases, namely the blocks 
\[
R_{j,r}=[q_{N_{r}+jr+1},q_{N_{r}+jr+2},\dots,q_{N_{r}+(j+1)r}], \quad 0\le j <(N_{r+1}-N_{r})/r.
\]
Let $R=[R_1,R_2,R_3,\dots,R_r]$ be a block of $r$ bases that equals the block $R_{j,r}$ for some $j$; and let ${B}_i={B}_i(R)$, $1\le i \le R_1R_2\dots R_r$, be the sequence of all the possible digit blocks of length $r$ such that $B_i<R$, arranged in lexicographical order (i.e., starting with $B_1=[0,0,\dots,0]$ then $B_2=[0,0,\dots,0,1]$ and so on, ending with $B_{R_1R_2\dots R_{r-1}} = [R_1-1,R_2-1,\dots,R_r-1]$). 

Let us define the number $x_Q\in [0,1)$ by its digits in the following way. For any fixed $R=[R_1,R_2,\dots,R_r]$, let $j_1$ be the smallest $j$ such that $R=R_{j,r}$, $j_2$ be the next smallest $j$ such that $R=R_{j,r}$, and so on. First, define the digits of $x_Q$ corresponding to the bases $R_{j_1,r}$ to be $B_1$, so that $E_{N_r+j_1r+i}(x) = 0$ for $1\le i \le r$. Then, define the digits corresponding to the basees $R_{j_2,r}$ be $B_2$, and so on, so the digits corresponding to $R_{j_i,r}$ will be $B_{i\pmod{R_1R_2\dots R_r}}$.

\section{Proof of Theorem \ref{thm:main}}

We focus first on showing that $x_Q$ is $Q$-normal.

Let $B=[b_1,b_2,\dots,b_k]$ be an arbitrary block of digits such that $Q_n(B) = \infty$. To show that $x$ is $Q$-normal, we must show that
\[
N_n^Q(B,x)= Q_n(B) (1+o(1)).
\]

Let $N_n^{Q*}(B,x)$ be defined similarly to $N_n^Q(B,x)$, but have it only count those appearances of $B$ which occur up to the $n$th place within the digits corresponding to a single $R_{j,r}$, and not beginning in some $R_{j,r}$ and terminating in a different $R_{j',r'}$. Likewise, let $Q_n^*(B)$ be defined by \[ \sideset{}{^*}\sum_{i=1}^n \frac{I_i(B)}{q_i q_{i+1} \dots q_{i+k-1}}, \] where the starred sum only runs over those $i$ for which $[q_i,q_{i+1},\dots,q_{i+k-1}]$ is a sub-block of $R_{j,r}$ for some $j,r$.

To prove that $x_Q$ is $Q$-normal, it suffices to show the following three asymptotic equalities:
\begin{align}
Q_n(B) &\sim Q_n^*(B),\label{eq:1}\\
N_n^Q(B,x) &= N_n^{Q*}(B,x)+o(Q_n(B)), \text{ and}\label{eq:2}\\
N_n^{Q*}(B,x) &\sim Q_n^* (B).\label{eq:3}
\end{align}

Let $n$ be a large integer, and let $r=r(n)$ be defined by $N_{r}<n\le N_{r+1}$.

\subsection{Proof of \eqref{eq:1}}

The difference $Q_n(B)-Q_n^*(B)$ is at most the sum
\[
\sum  \frac{1}{q_i q_{i+1} \dots q_{i+k-1}}
\] where the sum runs over all $i\le n$ such that the sub-blocks $[q_i,q_{i+1},\dots,q_{i+k-1}]$ start in some $R_{j_1,r_1}$ and end in another $R_{j_2,r_2}$. Each summand is at most $1/2^k$, so if we can show that the number of summands is $o(Q_n(B))$, we would have shown \eqref{eq:1}.

So let us count how many sub-blocks of $Q$ of length $k$ up to the $n$th place start in some $R_{j_1,r_1}$ and end in another $R_{j_2,r_2}$. Clearly every sub-block that starts before the $N_{k}$th place satisfies this condition. Each remaining sub-block occurs starting in one of the last $k-1$ places of a block of the form $R_{j,r}$ with $r\ge k$. Thus, at worst, the number of such sub-blocks is at most
 \begin{align*}
&(k-1) \left\lceil \frac{n-N_{r}}{r} \right\rceil + (k-1) \frac{N_{r}-N_{r-1}}{r-1} + \dots + (k-1) \frac{N_{k+1}-N_{k}}{k}+N_{k}\\
&\qquad \le (k-1)+\frac{k-1}{r} n + \frac{k-1}{r(r-1)} N_{r} + \frac{k-1}{(r-1)(r-2)} N_{r-1} + \dots +\frac{k-1}{(k+1)k} N_{k+1} \\
&\qquad\qquad+\frac{1}{k} N_{k}\\
&\qquad \le \frac{k-1}{r} n + (k-1)\sum_{i=2}^{r} \frac{N_i}{i(i-1)} + O_k(1).
\end{align*}
By definition, we have that $n_i/n_{i-1} \ge 5$, so that $N_i/N_{i-1} \ge 4$ for sufficiently large $i$. Thus, there exists a uniform constant $C$ such that 
\[
\frac{N_i}{i(i-1)} \le \frac{1}{C} \frac{N_{i+1}}{(i+1)i},
\]for all $i\ge 0$. So we have 
\[
\sum_{i=2}^{r} \frac{N_i}{i(i-1)} \ll \frac{N_{r}}{r(r-1)} \ll \frac{n}{r(r-1)}.
\]
Thus the number of these sub-blocks is at worst 
\begin{equation*}
\frac{k-1}{r}n + \frac{k-1}{r(r-1)}n + O_k(1) = O_k\left( \frac{n}{r} \right).
\end{equation*}
At this point, to show that this is $o(Q_n(B))$, we must show that $r \gg \log n/\log q(n)$.

Since we defined $r$ by $ n \le N_{r+1}$, we have
\[
n \le (q(n_{r+1})^2+1)^r+O(r).
\]
By taking logarithms, we obtain
\[
r \gg \frac{\log n}{\log q(n_{r+1})} 
\]
Since $n_{r+1}\ge N_{r+1}\ge n$ and $q(n)$ is a non-decreasing function, the desired asymptotic inequality follows.

\subsection{Proof of \eqref{eq:2}}

The difference $N_n^{Q}(B,x) - N_n^{Q*}(B,x)$ is at most the number of sub-blocks of $Q$ of length $k$ up to the $n$th place that  start in some $R_{j_1,r_1}$ and end in another $R_{j_2,r_2}$.  By the argument of the previous section, this difference is at most $o(Q_n(B))$.

\subsection{Proof of \eqref{eq:3}}\label{sec:prooflastasymp}

Consider a sub-block $R=[R_1,R_2,\dots,R_{r'}]$ with $r'\le r$ and an integer $i$ such that $1\le i \le r'-k+1$ and $B < [R_i, R_{i+1},\dots,R_{i+k-1}]$. Let $\mathcal{R}=R_1R_2\dots R_{r'}$ and $\mathcal{R}_i=R_i R_{i+1}\dots R_{i+k-1}$. For any $\mathcal{R}$ consecutive $j$'s for which $R=R_{j,r'}$, the corresponding digits of $x$ will run through all possible blocks of digits exactly once, and thus the digits of $B$ appear in the $i$th place of these sub-blocks exactly $\mathcal{R}/\mathcal{R}_i$ times.

Let $J_{R,n}$ denote the number of $j$ such that $R=R_{j,r'}$ with all the bases of $R_{j,r'}$ occurring before the $n$th place. (We will say that $R_{j,r'}$ occurs completely before the $n$th place.) By the argument of the previous paragraph, the number of times the digits $B$ occur in the $i$th place of the blocks $R_{j,r'}$ is 
\begin{align*}
\frac{\mathcal{R}}{\mathcal{R}_i} \left( \frac{J_{R,n}}{\mathcal{R}}+O(1)\right)&= \frac{1}{\mathcal{R}_i}J_{R,n} + O\left( \frac{\mathcal{R}}{\mathcal{R}_i}\right)\\
&= \frac{1}{\mathcal{R}_i} J_{R,n} + O(q(n)^{r'}),
\end{align*}
where the last equality comes from the fact that each base is at most $q(n)$.

Therefore, we have
\[
N_n^{Q*}(B,x) = \sum_{k\le r'\le r} \sum_{|R|=r'} \sum_{\substack{1 \le i \le r'-k+1 \\ B <   [R_i, R_{i+1},\dots,R_{i+k-1}]}} \left( \frac{1}{\mathcal{R}_i} J_{R,n} + O(q(n)^{r'}) \right)
\]
where the second sum runs over all $R$ such that $R=R_{j,r'}$ for some $R_{j,r'}$ that appears completely before the $n$th place. Let us treat the big-O term separately. We have
\begin{align*}
 \sum_{k\le r'\le r} \sum_{|R|=r'} \sum_{\substack{1 \le i \le r'-k+1 \\ B <   [R_i, R_{i+1},\dots,R_{i+k-1}]}} q(n)^{r'}& \le  \sum_{k\le r'\le r} \sum_{|R|=r'}  r' q(n)^{r'}\\
&\le   \sum_{k\le r'\le r} r q(n)^{2r'}\\
&\le r^2 q(n)^{2r}.
\end{align*}
Therefore,
\[
N_n^{Q*}(B,x) = \left( \sum_{k\le r'\le r} \sum_{|R|=r'} \sum_{\substack{1 \le i \le r'-k+1 \\ B <   [R_i, R_{i+1},\dots,R_{i+k-1}]}}  \frac{1}{\mathcal{R}_i} J_{R,n} \right) +  O(r^2 q(n)^{2r}).
\]

If we examine this triple sum carefully and recall the definition of $J_{R,n}$, we see that this is 
\[
\sum \frac{I_i(B)}{q_i q_{i+1}\dots q_{i+k-1}}
\]
where the sum runs over all $i$ such that $[q_i,q_{i+1},\dots, q_{i+k-1}]$ is a sub-block of some $R_{j,r'}$ that appears completely before the $n$th place. This sum is $Q_n^*(B)$ up to $O(r)$ (to account for the possibility that $n$ occurs in the middle of some sub-block $R_{j,r'}$).

Therefore,
\[
N_n^{Q*}(B,x) = Q_n^*(B)+O(r) +  O(r^2 q(n)^{2r}) =Q_n^*(B) +  O(r^2 q(n)^{2r})  .
\]

For sufficiently large $n$ (which in turn will give large $r$), we have
\[
r^2 q(n)^{2r} \ll \frac{(q(n)^2+1)^r}{r} \ll \frac{n}{r} = o(Q_n(B)).
\]
Since we already know that $Q_n(B)\sim Q_n^*(B)$, we therefore have 
\[
N_n^{Q*}(B,x) = Q_n^*(B) +  o(Q_n^*(B)),
\]
which completes the proof of $Q$-normality.

\subsection{Proof of $Q$-distribution normality}

Our goal now is to show that $(T_{Q,m}(x))_{m=0}^\infty$ is uniformly distributed. (The switch from labelling indicies by $n$ to labelling indices by $m$ is intentional.) We have that
\[
T_{Q,m}(x) = \sum_{i=1}^\infty \frac{E_{m+i}}{q_{m+1}q_{m+2}\dots q_{m+i}}
\]
where $(E_i)_{i=1}^\infty$ are the digits of $x_Q$.

Let $x_m$ be defined by
\[
x_m := \sum_{i=1}^{\ell(r(m))} \frac{E_{m+i}}{q_{m+1}q_{m+2}\dots q_{m+i}}.
\]
where $\ell(y)= \lfloor \sqrt{y}\rfloor$. 
Since 
\[
\left|x_m - T_{Q,m}(x)\right| \le \frac{1}{q_{m+1}q_{m+2}\dots q_{m+\ell(r(m))}} \le 2^{-\ell(r(m))},
\] which tends to $0$ with $m$, 
we have that $(T_{Q,m}(x))_{n=0}^\infty$ is uniformly distributed if and only if $(x_m)_{m=0}^\infty$ is.

Let $\mathcal{I}$ be some interval in $[0,1)$. To complete the proof of $Q$-distribution normality, we must show that
\[
\#\{0 \le m \le n : x_m \in \mathcal{I}\} = n \mathcal{I} (1+o(1)).
\]

As we did earlier, consider a block $R=[R_1,R_2,\dots, R_{r'}]$ with $r'\le r$ and an integer $i$ such that $1 \le i \le r'-\ell(r')+1$. Let $\mathcal{R} = R_1 R_2\dots R_{r'}$ and $\mathcal{R}_i = R_i R_{i+1}\dots R_{i+\ell(r')-1}$. Let $J_{R,n}$ denote the number of $j$ such that $R=R_{j,r'}$ with $R_{j,r'}$ occurring completely before the $ n$th place. 

Suppose that $m\le n$ and the $q_m$ appears in $Q$ at precisely the $i$th place of a sub-block $R_{j,r'}$. Then $x_m$ is a rational number with denominator $\mathcal{R}_i$. The number of distinct blocks of digits $B< [R_i, R_{i+1},\dots, R_{i+\ell(n)-1}]$ such that $x_m\in \mathcal{I}$ is $\mathcal{R}_i|\mathcal{I}|+O(1)$. And thus, by applying the same technique as in Section \ref{sec:prooflastasymp}, we see that the number of times $x_m \in \mathcal{I}$ with $m$ satisfying the above conditions is 
\[
\left(\frac{J_{R,n}}{\mathcal{R}_i}+O(q(n)^{r'})\right) \left( \mathcal{R}_i|\mathcal{I}|+O(1)\right) = J_{R,n}|\mathcal{I}| + O\left( \frac{J_{R,n}}{\mathcal{R}_i} \right) + O(q(n)^{r'}).
\]

For any fixed small $\epsilon>0$, let $r_\epsilon$ be an integer large enough so that $|\mathcal{I}|2^{-r_\epsilon} < \epsilon$ and $\ell(r')/r' < \epsilon$ for any $r' > r_{\epsilon}$. Then we have
\begin{align*}
\#\{0 \le m \le n : x_m \in \mathcal{I}\}&\ge \sum_{r_\epsilon \le r'\le r} \sum_{|R|=r'} \sum_{1\le i \le r'-\ell(r')+1} \left(J_{R,n}|\mathcal{I}| + O\left( \frac{J_{R,n}}{\mathcal{R}_i} \right) + O(q(n)^{r'})\right)\\
&=\sum_{r_\epsilon \le r'\le r} \sum_{|R|=r'} \sum_{1\le i \le r'-\ell(r')+1} \left(J_{R,n}|\mathcal{I}|(1+O(\epsilon)) + O(q(n)^{r'})\right).
\end{align*}
The inequality is due to not counting those $i$ for which $i\ge r'-\ell(r')+1$.

Again, by the work of Section \ref{sec:prooflastasymp}, we know that the sum over $O(q(n)^{r'})$ will be at most $O(n/r)=o(n)$. Thus,
\begin{align*}
\#\{0 \le m \le n : x_m \in \mathcal{I}\} &\ge o(n)+\sum_{r_\epsilon \le r'\le r} \sum_{|R|=r'} \sum_{1\le i \le r'-\ell(r')+1} \left(J_{R,n}|\mathcal{I}|(1+O(\epsilon))\right) \\
&= o(n) + O(N_{r_\epsilon}) + \sum_{1\le r'\le r} \sum_{|R|=r'} \sum_{1\le i \le r'-\ell(r')+1} \left(J_{R,n}|\mathcal{I}|(1+O(\epsilon))\right)\\
&= o(n) + \sum_{1\le r'\le r} \sum_{|R|=r'} (r'-\ell(r')) \left(J_{R,n}|\mathcal{I}|(1+O(\epsilon))\right).
\end{align*}
By the definition of $r_\epsilon$, we have $r'-\ell(r')=r'(1+O(\epsilon))$, so that
\begin{align*}
\#\{0 \le m \le n : x_m \in \mathcal{I}\} &\ge  o(n)  + \sum_{1\le r'\le r} \sum_{|R|=r'} \left(r' J_{R,n}|\mathcal{I}|(1+O(\epsilon))\right)\\
&= o(n) + n|\mathcal{I}|(1+O(\epsilon)) + O(r)\\
&= n|\mathcal{I}|(1+O(\epsilon)+o(1))
\end{align*}
where again the $O(r)$ term comes from the fact that $n$ could be in the middle of some term $R_{j,r}$. Since $\epsilon$ was arbitrary, the desired result follows.

\section{Computability}

\begin{proof}[Proof of Theorem \ref{thm:comp}]
Since $(q(n))$ is a computable sequence of integers, the sequence $((q(n)^r+1)^r)_{n=1}^\infty$ is also a computable sequence of integers for a fixed integer $r$. Create a Turing machine $M : \N \to \N$ such that $M(r) = n_r$ as follows. Consider the Turing machine $L :  \N \times \N \to \br{0,1}$ such that $L(x,y) = 1$ if when $x \leq y$ and $0$ otherwise. For input $r$ have $M$ at step $n$ output $n$ if $L((q(n)^2+1)^r, n) = 1$ and halt, otherwise increment $n$ by $1$. This process halts because $n_r$ exists, and $M(r) = n_r$. Thus the sequence $(n_r)_{r=1}^\infty$ is a computable sequence of integers.

To see the sequence $(N_r)_{r=1}^\infty$ is a computable sequence of integers, consider the Turing machine $I : \N \times \N \to \N$ such that $I(n,r)$ is the greatest integer less than $n_r$ such that $I(n,r)-n \equiv 0 \mod r$. Construct $I$ as follows. Consider the Turing machine $Rem_r : \N \to \br{0,1}$ such that $Rem_r(n) \equiv n \mod r$ with $0 \leq Rem_r(n) < r$. Have $I$ first compute $n-Rem_r(n)$, and at step $k$ check if $L(n-Rem_r(n)+kr, n_r) = 0$. If $L(n-Rem_r(n)+kr, n_r) = 0$, have $I$ output $n-Rem_r(n)+(k-1)r$, otherwise increment $k$ by 1. Finally construct the Turing machine $N: \N \to \N$ that on input $r$ computes $I(I(\cdots I(I(0,1),2)\cdots r-1), r)$. Then $N(r) = N_r$, so $(N_r)_{r = 1}^\infty$ is a computable sequence of integers.

Now construct the Turing machine $E : \N \to \N$ with $E(n) = E_n$ as follows. First make the machine $r : \N \to \N$ such that $r(n)$ is the integer $r$ such that $N_r \leq n < N_{r+1}$. Such a machine exists because the sequence of integers $(N_r)$ is computable and the order relation on the integers is a computable relation. Construct a Turing machine $J_r : \N^r \times \N \to \N$ such that $J([R_1, R_2, \cdots, R_r], n)$ is the number of times the block $[R_1, R_2, \cdots , R_r]$ occurs a position $t$ in $Q$ with $t \equiv N_r \mod r$, $t \geq N_r$, and $t \leq n$. Create a Turing machine $B_r : \N^r \times \N \to \N^r$ such that $B_r([R_1, R_2, \cdots, R_r], i)$ is the $i$th block $B$ in the lexicographic ordering on $\N^r$ with $B < [R_1, R_2, \cdots, R_r]$. Finally let $R_r: \N \to \N^r$ be the Turing machine with $R_r(n) = [q_{N_r + jr+1}, q_{N_r+jr+2}, \cdots, q_{N_r+(j+1)r}]$ such that $N_r+jr+1 \leq n \leq N_r+(j+1)r$. Then $E(n)$ is the $n - N_{r(n)} \mod r$th element of $B_{r(n)}(R_{r(n)}(n),J_{r(n)}(R_{r(n)}(n),n))$. Thus the sequence $(E_n)$ is a computable sequence. Since $Q$ is also a computable sequence of integers, the real number $x_Q = \sum_{i=1}^\infty \frac{E_i}{q_1 \cdots q_i}$ that was constructed in Section \ref{sec:construction} is a computable real number.
\end{proof}

Using this theorem, we can now give computable examples of numbers that are normal of one type but not another, as in \cite{AireyManceHDDifference} and  \cite{ppq1}. We will need the following definition and theorem from \cite{ppq1}.

Let $(P,Q)$ be a pair of basic sequences and suppose that $x = E_0. E_1 E_2 \cdots $ w.r.t. $P$. We define
$$
\psi_{P,Q}(x) : = \sum_{n=1}^\infty \frac{\min \br{E_n, q_n-1}}{q_1 \cdots q_n}.
$$

\begin{thm}\label{thm:psipq}
Suppose that  $Q_1 = (q_{1,n}), Q_2 = (q_{2,n}), \cdots, Q_j = (q_{j,n})$ are basic sequences and infinite in limit. Set
$$
\Psi_j(x) = (\psi_{Q_{j-1}, Q_j} \circ \psi_{Q_{j-2},Q_{j-1}} \circ \cdots \circ \psi_{Q_1, Q_2})(x).
$$
If $x = E_0.E_1E_2 \cdots$ w.r.t. $Q_1$ satisfies $E_n < \min_{2 \leq r \leq j}(q_{r,n}-1)$ for infinitely many $n$, then for every block $B$
$$
N_{n}^{Q_j}(B,\Psi_j(x)) = N_{n}^{Q_1}(B,x) + O(1).
$$
\end{thm}

We now state the following three theorems.

\begin{thm}
If $Q$ is slowly growing, infinite in limit and the sequences $(q_n)$ and $(q(n))$ are computable sequences of integers, then there is a computable real number in $\NQ \setminus \DNQ$
\end{thm}
\begin{proof}
Let $p_n = \max\br{\lfloor \log q_n \rfloor, 2}$ and set $P = (p_n)$. By \reft{comp} there is a computable real number $x_Q \in \NQ$. Put $y = (\psi_{P,Q} \circ \psi_{Q,P})(x)$. Then $y$ is $Q$-normal by \reft{psipq} but $T_{Q,n}(y) \to 0$ so $y$ is not $Q$-distribution normal. Furthermore, $E_{Q,n}(y) = \max \br{ E_{Q,n}(x), \lfloor \log q_n \rfloor, 2}$ which is a computable sequence of integers. Therefore $y$ is a computable real number.
\end{proof}

\begin{thm}
If $Q$ is slowly growing, infinite in limit, and the sequences $(q_n)$ and $(q(n))$ are computable sequences of integers, then there is a computable real number in $\RNQ \setminus \NQ$.
\end{thm}
\begin{proof}
Let $p_n = \max \br{ \floor{q_n/2}, 2}$ and set $P= (p_n)$. The basic sequence $P$ clearly has the same properties at $Q$. Let $x_Q$ be a computable real number in $\mathscr{N}(P)$ and set $y = \psi_{P,Q}(x)$. The real number $y$ is clearly computable, and by the calculations in \cite{ppq1} is in $\RNQ \setminus \NQ$.
\end{proof}

To prove the next result we will need the following definition and lemma. For a sequence of real numbers $X = (x_n)$ with $x_n \in [0,1)$ and an interval $I \subseteq [0,1]$,  define $A_n(I,X) = \# \{i\leq n: x_i \in I \}$.
We quote the following from \cite{KuN}.

\begin{defn}
Let $X = \pr{x_1, \cdots , x_N}$ be a finite sequence of real numbers. The number $$D_N = D_N(X) = \sup_{0 \leq \alpha \leq \beta \leq 1} \left | \frac{A_N([\alpha, \beta), X)}{N} - (\beta - \alpha) \right |$$ is called the {\it discrepancy} of the sequence $\omega$.
\end{defn}

It is well known that a sequence $X$ is uniformly distributed mod $1$ if and only if $D_N(X) \to 0$.

\begin{lem}\label{lem:DiscKuN}
Let $x_1, x_2, \cdots, x_N$ and $y_1, y_2, \cdots, y_N$ be two finite sequences in $[0,1)$.  Suppose $\epsilon_1, \epsilon_2, \cdots, \epsilon_N$ are non-negative numbers such that $|x_n-y_n| \leq \epsilon_n$ for $1 \leq n \leq N$.  Then, for any $\epsilon \geq 0$, we have
$$
|D_N(x_1,\cdots,x_N)-D_N(y_1,\cdots,y_N)| \leq 2\epsilon+\frac {\overline{N}(\epsilon)}{N},
$$
where $\overline{N}(\epsilon)$ denotes the number of $n$, $1 \leq n \leq N$, such that $\epsilon_n>\epsilon$.
\end{lem}

We can now prove the following theorem
\begin{thm}
If $Q$ is infinite in limit and computable and the sequence $(L_n)$ defined in the proof of Theorem \ref{thm:comp} is computable, then there is a computable real number in $\RNQ \cap \DNQ \setminus \NQ$.
\end{thm}
\begin{proof}
Let $P = (p_i)$ with $p_i = \lfloor \log i \rfloor + 2$. Note that $P$ is slowly growing, computable, and the sequence $(p(n))$ is computable, so there is a computable real number $\xi \in \mathscr{N}(P)$ with $\xi = .F_1 F_2 \cdots \text{w.r.t } P$. Fix a computable sequence of real numbers $X = (x_n)$ that is uniformly distributed modulo 1 (for example the Farey sequence). Define the sequences
\begin{align*}
&\nu_n = \min\br{t : \frac{\sum_{i=0}^{n-1} \log q_{L_{n-1}+i}}{\sum_{i=0}^{j-L_{n-1}-1} \log q_{L_{n-1}+i}} < \frac{1}{n}, \forall j\geq t};\\
&\upsilon_{n,k} = \min \br{t : \frac{Q_n(B)}{\sum_{i=1}^j P_{i-k+1}(B) } < \frac{1}{n}, \forall j \geq t \text{ and blocks } B \text{ of length } k}; \\
&L_0 = 0; \\
&L_n = \max \br{ \min\br{t : \log(q_j) > n, \forall j \geq t}, L_{n-1} + n^2, L_{n-1}+\nu_n, \max_{k \leq n} \br{\upsilon_{n,k}}}
\end{align*}
and set $i(n) = \max \{j : L_j \leq n \}$. The sequence $(i(n))$ is computable since $(L_n)$ is a computable sequence. Note that $\nu_n$ and $\upsilon_{n,k}$ exist since $Q$ is infinite in limit and $P$ is fully divergent.
Define the set 
$$
S = \bigcup_{n=1}^\infty \{L_n, L_n+1, \cdots, L_n+n-1\}.
$$
Note that this set has density $0$ since 
$$
\frac{\sum_{i=1}^n i}{\sum_{i=1}^n L_{i}- L_{i-1}} \leq \frac{\sum_{i=1}^n i}{\sum_{i=1}^n i^2} \to 0 \hbox{ as $n$ goes to infinity.}
$$
Define the sequence
$$
E_n=
\begin{cases}
F_{n-L_i} &\text{if } n \in [L_i, L_i+1, \cdots, L_i+i]\\
\max \br{\floor{x_n q_n}, \ceil{\log i(n)}} & \text{otherwise}  
\end{cases}
$$
We claim the real number $x = \sum_{n=1}^\infty \frac{E_n}{q_1 \cdots q_n}$ is in $\RNQ \cap \DNQ \setminus \NQ$. Let $B$ be a block of length $k$. Note that by the definition of $L_n$, there are only finitely many values $n \in \N \setminus S$ such that $B$ occurs at position $n$ in the $Q$-Cantor series expansion of $x$. This is because all digits $E_n$ with $n \in \N \setminus S$ must be at least $\ceil{\log i(n)}$ and since $i(n)$ tends to infinity as $n$ does. Thus if $m$ is the maximum digit for the block $B$, we have that for $n \in \N \setminus S$ with $i(n) > m$ that $E_n > m$. Thus $N_n^Q(B,x) = \sum_{i=1}^{i(n)} N_{i-k+1}^P(B, \xi) + O(1)$. So for any two blocks $B_1$ and $B_2$ of length $k$, we have
\begin{align*}
\lim_{n \to \infty} \frac{N_n^Q(B_1,x)}{N_n^Q(B_2, x)} &= \lim_{n \to \infty} \frac{\sum_{i=1}^{i(n)} N_{i-k+1}^P(B_1, \xi) + O(1)}{\sum_{i=1}^{i(n)} N_{i-k+1}^P(B_2, \xi) + O(1)} \\
&= \lim_{n \to \infty} \frac{N_{n-k+1}^P(B_1, \xi)}{N_{n-k+1}^P(B_2, \xi)} = 1.
\end{align*}
Thus $x \in \RNQ$.

Consider the sequence $Y = \pr{\frac{E_n}{q_n}}$. For $n \in \mathbb{N} \backslash S$, we have $\left |\frac{E_n}{q_n} - x_n \right | < \frac{1}{q_n}$, which tends to $0$ as $n$ goes to infinity. We therefore have for $\epsilon>0$ that $\overline{N}(\epsilon) = O(1) + \# S \cap \{1, \cdots, N\}$. Thus by Lemma \ref{lem:DiscKuN}
$$
\left |D_N(X) - D_N(Y) \right | < 2 \epsilon + \frac{O(1)}{N} + \frac{\# S \cap \{1, \cdots, N\}}{N} < 3 \epsilon
$$
if $N$ is sufficiently large. Since the inequality holds for all $\epsilon>0$, we have that $\pr{\frac{E_n}{q_n}}$ is uniformly distributed mod 1. Thus $x \in \DNQ$.

Note that
$$
\lim_{n \to \infty} \frac{N_n^Q(B,x)}{\sum_{i=1}^{i(n)} P_{i-k+1}(B) } = 1.
$$
However, 
$$
\lim_{n \to \infty} \frac{Q_n(B)}{\sum_{i=1}^{i(n)} P_{i-k+1}(B) } = 0
$$
by the definition of $L_n$, so $x \not \in \NQ$.

Furthemore, the sequence $E_n$ is computable because the sequences $(F_n)$, $(L_n)$ and $(i(n))$ are all computable. Thus $x$ is a computable real number in $\RNQ \cap \DNQ \setminus \NQ$.

\end{proof}


\bibliographystyle{amsplain}
\providecommand{\bysame}{\leavevmode\hbox to3em{\hrulefill}\thinspace}
\providecommand{\MR}{\relax\ifhmode\unskip\space\fi MR }
\providecommand{\MRhref}[2]{%
  \href{http://www.ams.org/mathscinet-getitem?mr=#1}{#2}
}
\providecommand{\href}[2]{#2}



\end{document}